\documentclass{amsart}
\usepackage{amssymb,amsthm,amstext,amsmath,amscd,latexsym,amsfonts}
\usepackage[all]{xy}
	
\newtheorem{theorem}{Theorem}[section]
\newtheorem{lemma}[theorem]{Lemma}

\newtheorem{proposition}[theorem]{Proposition}
\theoremstyle{definition}
\newtheorem{definition}[theorem]{Definition}

\theoremstyle{remark}
\newtheorem{remark}[theorem]{Remark}
\numberwithin{equation}{section} 
\def\ulL{{\underline{L}}}
\def\ulM{{\underline{M}}}
\def\ulN{{\underline{N}}}
\def\ulZ{{\underline{Z}}}
\def\ulX{{\underline{X}}}
\def\ulY{{\underline{Y}}}
\def\ulW{{\underline{W}}}
\def\ulU{{\underline{U}}}

\def\ulf{{\underline{f}}}
\def\ulg{{\underline{g}}}
\def\ulh{{\underline{h}}}
\def\ulw{{\underline{w}}}

\def\SCMR{{\underline{\C}}}
\def\SHom{{\underline{\mathrm{Hom}}}}
\def\SHomR{{\underline{\mathrm{Hom}}_R}}

\def\modR{\mathrm{mod} (R)}

\def\CMR{\C}
\def\ind {\mathrm{ind}}

\def\GR{\mathrm{G}(\CMR )}

\def\EXR{\mathrm{EX}(\CMR )}
\def\ARR{\mathrm{AR}(\C )}

\def\Hom{\mathrm{Hom}}
\def\HomR{\mathrm{Hom}_R}

\def\Ext{\mathrm{Ext}}
\def\Tor{\mathrm{Tor}}
\def\Im{\mathrm{Im}}
\def\tr{\mathrm{tr}}

\def\Z{\mathbb Z}
\def\m{\mathfrak m}
\def\p{\mathfrak p}

\def\C{\mathcal C}

\begin{document}
%%%%%%%%%%%%%%%%%%%%%%%%%%%%%%%%%%%%%%%%%%%%%%%%%%%%
%%%%%%%%%%%%%%%%%%%%%%%%%%%%%%%%%%%%%%%%%%%%%%%%%%%%
%%%%%%%%%%%%%%%%  Title %%%%%%%%%%%%%%%%%%%%%%%%%%%%%%%%%
%%%%%%%%%%%%%%%%%%%%%%%%%%%%%%%%%%%%%%%%%%%%%%%%%%%%
%%%%%%%%%%%%%%%%%%%%%%%%%%%%%%%%%%%%%%%%%%%%%%%%%%%%
\title[Relations for Grothendieck groups of Gorenstein rings]{Relations for Grothendieck groups of Gorenstein rings}
%%%%%%%%%%%%%%%%%%%%%%%%%%%%%%%%%%%%%%%%%
%%%%%%%%%%%%%%%%%%%%%%%%%%%%%%%%%%%%%%%%%
%%%%%%%% Informations of Authors %%%%%%%%
%%%%%%%%%%%%%%%%%%%%%%%%%%%%%%%%%%%%%%%%%
%%%%%%%%%%%%%%%%%%%%%%%%%%%%%%%%%%%%%%%%%
\author{Naoya Hiramatsu}
\address{Department of general education, Kure National College of Technology, 2-2-11, Agaminami, Kure Hiroshima, 737-8506 Japan}
\email{hiramatsu@kure-nct.ac.jp}
\thanks{This work was partly supported by JSPS Grant-in-Aid for Young Scientists (B) 15K17527. }
%%%%%%%%%%%%%%%%%%%%%%%%%%%%%%%%%%%%%%%%%
%%%%%%%%%%%%%%%%%%%%%%%%%%%%%%%%%%%%%%%%%
%%%%%%%%%%%% General info %%%%%%%%%%%%%%%
%%%%%%%%%%%%%%%%%%%%%%%%%%%%%%%%%%%%%%%%%
%%%%%%%%%%%%%%%%%%%%%%%%%%%%%%%%%%%%%%%%%
\date{\today}

\maketitle
%%%%%%%%%%%%%%%%%%%%%%%%%%%%%%%%%%%%%%%%%
%%%%%%%%%%%%%%%%%%%%%%%%%%%%%%%%%%%%%%%%%
%%%%%%%%%%%%%%%%%%%%%%%%%%%%%%%%%%%%%%%%%
%%%%%%%%%%%%% Abstract %%%%%%%%%%%%%%%%%%
%%%%%%%%%%%%%%%%%%%%%%%%%%%%%%%%%%%%%%%%%
%%%%%%%%%%%%%%%%%%%%%%%%%%%%%%%%%%%%%%%%%
%%%%%%%%%%%%%%%%%%%%%%%%%%%%%%%%%%%%%%%%%
\begin{abstract}
We consider the converse of the Butler, Auslander-Reiten's Theorem which is on the relations for Grothendieck groups. 
We show that a Gorenstein ring is of finite representation type if the Auslander-Reiten sequences generate the relations for Grothendieck groups. 
This gives an affirmative answer of the conjecture due to Auslander.     
\end{abstract}
%%%%%%%%%%%%%%%%%%%%%%%%%%%%%%%%%%%%%%%%%
%%%%%%%%%%%%%%%%%%%%%%%%%%%%%%%%%%%%%%%%%
%%%%%%%%%%%%%%%%%%%%%%%%%%%%%%%%%%%%%%%%%
%%%%%%%%%%%%%%%%%%%%%%%%%%%%%%%%%%%%%%%%%
\section{Introduction}
Throughout this section, $(R, \m )$ denote a commutative Cohen-Macaulay complete ring with the residue field $k$. 
All $R$-modules are assumed to be finitely generated. 
We say that an $R$-module $M$ is Cohen-Macaulay if
$$
\Ext _{R}^{i}(k, M) = 0 \quad \text{for all } i < \dim R. 
$$
We denote by $\modR$ the category of $R$-modules with $R$-homomorphisms and by $\C$ the full subcategory of $\modR$ consisting of all Cohen-Macaulay $R$-modules.  

Let $K_0 (\C )$ be a Grothendieck group of $\C$. 
Since $K_0 (\CMR) = K_0 (\modR )$, it is important to investigate $K_0 (\C )$ for the study of K-theory of $\modR$. 
Set $\GR =\bigoplus _{X \in \ind \C} \ \Z \cdot [X]$, which is a free abelian group generated by isomorphism classes of indecomposable objects in $\CMR$. 
We denote by $\EXR$ a subgroup of $\GR$ generated by
$$
\{  [X] + [Z] - [Y] \ | \text{there is an exact sequence }0 \to Z \to Y \to X \to 0 \text{ in } \CMR \} .
$$ 
We also denote by $\ARR$ a subgroup of $\GR$ generated by 
$$
\{ [X] + [Z] - [Y] \ | \text{there is an AR sequence }0 \to Z \to Y \to X \to 0 \text{ in } \CMR \} . 
$$
On the relation for Grothendieck groups, Butler\cite{B81}, Auslander-Reiten\cite{AR86}, and Yoshino\cite{Y90} prove the following theorem.

\begin{theorem}\cite{B81, AR86, Y90}\label{BARY}
If $R$ is of finite representation type then $\EXR = \ARR$. 
\end{theorem}

Here we say that $R$ is of finite representation type if there are only a finite number of isomorphism classes of indecomposable Cohen-Macaulay $R$-modules.   

In this note we consider the converse of Theorem \ref{BARY}. 
Actually we shall show the following theorem. 
 
\begin{theorem}\label{Main theorem}
Let $R$ be a complete Gorenstein local ring with an isolated singularity and with algebraically closed residue field. 
If $\EXR = \ARR$, then $R$ is of finite representation type.  
\end{theorem}

Auslander conjectured the converse of Theorem \ref{BARY} is true. 
It has been proved by Auslander\cite{A84} for Artin algebras and by Auslander-Reiten\cite{AR86} for complete one dimensional domain. 
Our theorem gives an affirmative answer to his conjecture for the case of complete Gorenstein local rings with an isolated singularity. 

%%%%%%%%%%%%%%%%%%%%%%%%%%
%%%%% section 1%%%%%%%%%%%%%%%%
%%%%%%%%%%%%%%%%%%%%%%%%%%
\section{Proof of Theorem \ref{Main theorem}}\label{result}

In the rest of the note, we always assume that $(R, \m)$ is a complete Gorenstein local ring  with the residue field $k$, where $k$ is an algebraically closed field. 
For the category of Cohen-Macaulay $R$-modules $\CMR$, we denote by $\SCMR$ the stable category of $\CMR$. 
The objects of $\SCMR$ are the same as those of $\CMR$, and the morphisms of $\SCMR$ are elements of $\SHom _R(M, N) = \Hom _R(M, N)/P(M, N)$ for $M, N \in \SCMR$, where $P(M, N)$ denote the set of morphisms from $M$ to $N$ factoring through free $R$-modules. 
Since $R$ is complete, $\CMR$, hence $\SCMR$, is a Krull-Schmidt category. 
For $M \in \CMR$ we denote it by $\ulM$ to indicate that it is an object of $\SCMR$. 
For a finitely generated $R$-module $M$, take a free presentation 
$$
\cdots \rightarrow F_1 \xrightarrow{d} F_0 \rightarrow M \rightarrow 0.
$$
We denote $\Im \ d$ by $\Omega M$, which is called a (first) syzygy of $M$.  
And we also denote by $\tr M$ the cokernel $F_0 ^{\ast}\xrightarrow{d^{\ast}}F_1^{\ast}$ where $(-)^{\ast} = \HomR (-, R)$.

First of all we prepare a key lemma.

\begin{lemma}\label{key}
There exists $X \in \CMR$ such that $\SHomR (M, X) \not= 0$ for all $M \in \CMR$ with $\ulM \not= \underline{0}$ in $\SCMR$. 
\end{lemma}

\begin{proof}
Take a Cohen-Macaulay approximation of the residue field $k$ as $X$;
\begin{equation}\label{CM approx.}
0 \to Y \to X \to k \to 0.
\end{equation}
Then we shall show $X$ satisfies the assertion of the lemma. 

Let $M$ be a non free Cohen-Macaulay module, that is, $\ulM \not= \underline{0}$ in $\SCMR$. 
Apply $\HomR (M, -)$ to the sequence (\ref{CM approx.}), we have the commutative diagram with exact rows, where the vertical arrows are natural surjections;  
$$
\begin{CD}
@. @.   \SHomR (M, X) @>>> \SHomR (M, k)@>>> 0   \\ 
@. @.@AAA @ AAA  @. \\ 
0 @>>> \HomR (M, Y)  @>>> \HomR (M, X) @>>> \HomR (M, k) @>>> 0.  \\
\end{CD}
$$
Assuming $\SHomR (M, X) = 0$, we have $\SHomR (M, k) = 0$ from this diagram. 
On the other hand, since $\SHomR (M, k) = \Tor _R ^1 (\tr M, k)$ (\cite[Lemma 3.9]{Y90}), this implies that $\tr M$ is free (\cite[\S 19 Lemma 1(i)]{M89}), so that $M$ is free. 
This is a contradiction.   
\end{proof}

The stable category $\SCMR$ has a structure of a triangulated category since $R$ is Gorenstein (cf. \cite{Ha88}). 
By the definition of a triangle, $ \ulL \to \ulM \to \ulN \to \ulL [1]$ is a triangle in $\SCMR$ if and only if there is an exact sequence $0 \to L \to M' \to N \to 0$  in  $\CMR$ with $\ulM ' \cong \ulM$ in $\SCMR$. 
To prove our theorem, we use a theory of Auslander-Reiten (abbr. AR) triangles. 
The notion of AR triangles is a stable analogy of AR sequences.

\begin{definition}\cite[Chapter I, \S 4]{Ha88} 
We say that a triangle $\ulZ \to \ulY \xrightarrow{\ulf} \ulX \xrightarrow{\ulw} \ulZ [1] $ in $\SCMR$ is an AR triangle ending in $\ulX$ (or starting from $\ulZ$) if it satisfies
\begin{itemize}
\item[(1)] $\ulX$ and $\ulZ$ are indecomposable. 

\item[(2)] $\ulw \not=0$. 

\item[(3)] If $\ulg : \ulW \to \ulX$ is not a split epimorphism, then there exists $\ulh : \ulW \to \ulY$ such that $\ulg = \ulf \circ \ulh$.  
\end{itemize}
\end{definition}

\begin{remark}\label{AR triangle}
Let $0 \to Z \to Y \xrightarrow{f} X \to 0 $ be an AR sequence in $\CMR$. 
Then $\ulZ \to \ulY \xrightarrow{\ulf} \ulX \to \ulZ [1] $ is an AR triangle in $\SCMR$. 
See \cite[Proposition 2.2]{Hi15} for example. 
\end{remark}

We say that $(R, \m)$ is an isolated singularity if each localization $R_{\p }$ is regular for each prime ideal $\p$ with $\p \not= \m$. 
Note that if $R$ is an isolated singularity, $\CMR$ admits AR sequences (cf. \cite[Theorem 3.2]{Y90}). 
Hence $\SCMR$ admits AR triangles (Remark \ref{AR triangle}). 
We also note that since we have the isomorphism $\SHomR (M, N) \cong \Tor _1 ^R (\tr M, N)$ for finitely generated $R$-modules $M$ and $N$, one can show that $\mathrm{length} _R (\SHomR (M, N) )$ is finite for $M$, $N \in \SCMR$ if $R$ is an isolated singularity. 
When $U$ is indecomposable in $\CMR$ then denote by $\mu (\ulU, \ulX)$ the multiplicity of $\ulU$ as a direct summand of $\ulX$.

\begin{proposition}\label{prop of AR triangles}\cite[Proposition 3.1]{W13}\cite[Proposition 2.14 (1)]{Hi15}
Let $R$ be an isolated singularity and let $\ulZ \xrightarrow{\ulg} \ulY \xrightarrow{\ulf} \ulX \to \ulZ [1]$ be an AR triangle in $\SCMR$. 
Then the following equality holds for each indecomposable $U \in \CMR$: 
$$
[\ulU, \ulX ] + [\ulU, \ulZ ] - [\ulU, \ulY ] = \mu (\ulU, \ulX) + \mu (\ulU, \ulX [-1]). 
$$
Here $[\ulU, \ulX] $ is an abbreviation of $\mathrm{length} _R (\SHomR (M, N) )$. 
\end{proposition}

\noindent
{\it Proof of Theorem \ref{Main theorem}. } 
Let $X$ be the module that satisfies the conditions as in Lemma \ref{key}. 
Take the syzygy of $X$. 
$$
0 \to \Omega X \to P \to X \to 0. 
$$
By the assumption, since $\EXR = \ARR$, we have the equality in $\GR$, 
$$
[X] + [\Omega X] - [P] = \sum _{M \in \mathrm{ind} \CMR }^{finite} a_{M, X}([M] + [\tau M] - [E_M]), 
$$
where $[M] + [\tau M] - [E_M]$ come from AR sequences $0 \to \tau M \to E_M \to M \to 0$. 
The equality yields that
\begin{equation}\label{equation}
[\ulU, \ulX \oplus \underline{\Omega X}] = \sum _{M \in \mathrm{ind} \CMR }^{finite} a_{M, X}([\ulU, \ulM] + [\ulU, \underline{\tau M}] - [\ulU, \underline{E_M}])
\end{equation}
for each $U \in \CMR$. 
Since $\underline{\tau M} \to \underline{E_M} \to \ulM \to \underline{\tau M} [1]$ are AR triangles (Remark \ref{AR triangle}), by Proposition \ref{prop of AR triangles}, we see that there are only a finite number of indecomposable modules in $\CMR$ that makes the RHS in (\ref{equation}) non-zero, so is LHS. 
By Lemma \ref{key}, we conclude that $\SCMR$, hence $\CMR$ is of finite representation type. 
\qed

%%%%%%%%%%%%%%%%%%%%%%%%%%%%%%%%%%%%%%%%%%%%%%%%%%%%%%%
%%%%%%%%%%%% Acknowledgment %%%%%%%%%%%%%%%%%%%%%%%%%%%
%%%%%%%%%%%%%%%%%%%%%%%%%%%%%%%%%%%%%%%%%%%%%%%%%%%%%%%
\section*{Acknowledgments}
The author express his deepest gratitude to Osamu Iyama and Yuji Yoshino for valuable discussions and helpful comments. 
The author also thank the referee for his/her careful reading.
%%%%%%%%%%%%%%%%%%%%%%%%%%%%%%%%%%%%%%%%%%%%%%%%%%
%%%%%%%%%%%%%%%%%%%%%%%%%%%%%%%%%%%%%%%%%%%%%%%%%%
%%%%%%%%%%%%%%%%%%%%%%%%%%%%%%%%%%%%%%%%%%%%%%%%%%
%%%%%%%%%% References %%%%%%%%%%%%%%%%%%%%%%%%%%%%
%%%%%%%%%%%%%%%%%%%%%%%%%%%%%%%%%%%%%%%%%%%%%%%%%%
%%%%%%%%%%%%%%%%%%%%%%%%%%%%%%%%%%%%%%%%%%%%%%%%%%
%%%%%%%%%%%%%%%%%%%%%%%%%%%%%%%%%%%%%%%%%%%%%%%%%%

\end{document}